\documentclass[12pt]{amsart}
\allowdisplaybreaks
\usepackage[english]{babel}
\usepackage[pdftex,paper=a4paper,portrait=true,textwidth=450pt,textheight=675pt,tmargin=3cm,marginratio=1:1]{geometry}
\usepackage{amsfonts}
\usepackage[dvips]{graphics}
\usepackage[colorinlistoftodos]{todonotes}
\usepackage{amsmath}
\usepackage{amsthm}
\usepackage{amssymb}
\usepackage{bbm}
\usepackage{cancel}
\usepackage{color}
\usepackage[curve]{xypic}
\usepackage{graphicx}
\newtheorem{theorem}{Theorem}[section]

\newtheorem{proposition}[theorem]{Proposition}
\newtheorem{lemma}[theorem]{Lemma}
\newtheorem{conjecture}[theorem]{Conjecture}

\theoremstyle{definition}
\newtheorem{definition}[theorem]{Definition}

\newtheorem{remark}[theorem]{Remark}

\theoremstyle{property}

\DeclareFontFamily{OT1}{rsfs}{}
\DeclareFontShape{OT1}{rsfs}{n}{it}{<-> rsfs10}{}
\DeclareMathAlphabet{\curly}{OT1}{rsfs}{n}{it}

\newcommand\LL{\mathbb L}

\renewcommand\O{\mathcal O}
\newcommand\PP{\mathbb P}

\newcommand\cE{\mathcal E}

\newcommand\F{\mathcal F}

\newcommand\C{\mathbb C}

\newcommand\sfZ{\mathsf Z}

\newcommand\Q{\mathbb Q}

\newcommand\Z{\mathbb Z}

\newcommand\Coh{\mathrm{Coh}}

\newcommand\tor{\mathrm{tor}}

\newcommand\SU{\mathrm{SU}}

\newcommand\vd{\mathrm{vd}}

\newcommand\vir{\mathrm{vir}}

\newcommand\VW{\mathrm{VW}}

\newcommand\td{\mathrm{td}}

\newcommand\rk{\operatorname{rk}}

\newcommand\tr{\operatorname{tr}}

\newcommand\ch{\operatorname{ch}}

\newcommand\Hom{\operatorname{Hom}}
\renewcommand\hom{\mathcal{H}{\it{om}}}
\newcommand\Ext{\operatorname{Ext}}
\newcommand\ext{\curly Ext}

\newcommand\Pic{\operatorname{Pic}}

\newcommand\Spec{\operatorname{Spec}\,}
\newcommand\Hilb{\operatorname{Hilb}}

\newcommand\Sym{\operatorname{Sym}}

\newcommand\mdot{{\scriptscriptstyle\bullet}}

\newcommand\INTO{\ar@{^{(}->}[r]}

\DeclareRobustCommand{\SkipTocEntry}[4]{}

\begin{document}
\title[Twisted sheaves and $\SU(r) / \Z_r$ Vafa-Witten theory]{Twisted sheaves and $\SU(r) / \Z_r$ Vafa-Witten theory}
\author[Y.~Jiang and M.~Kool]{Yunfeng Jiang and Martijn Kool}
\maketitle

\vspace{-1cm}

\begin{abstract}
The $\mathrm{SU}(r)$ Vafa-Witten partition function, which virtually counts Higgs pairs on a projective surface $S$, was mathematically defined by Tanaka-Thomas. On the Langlands dual side, the first-named author recently introduced virtual counts of Higgs pairs on $\mu_r$-gerbes. In this paper, we instead use Yoshioka's moduli spaces of twisted sheaves. Using Chern character twisted by rational $B$-field, we give a new mathematical definition of the $\SU(r) / \Z_r$ Vafa-Witten partition function when $r$ is prime. Our definition uses the period-index theorem of de Jong. 

$S$-duality, a concept from physics, predicts that the $\mathrm{SU}(r)$ and $\mathrm{SU}(r) / \Z_r$ partition functions are related by a modular transformation. We turn this into a mathematical conjecture, which we prove for all $K3$ surfaces and prime numbers $r$.
\end{abstract}

\section{Introduction} 

In physics, $S$-duality relates the partition function of $N=4$ supersymmetric Yang-Mills theory with gauge group $G$ on a 4-manifold $M$ to the partition function of the theory with Langlands dual gauge group $^{L}G$. In 1994, C.~Vafa and E.~Witten \cite{VW} considered a topological twist of this theory. Taking $G = \mathrm{SU}(r)$ and $M$ underlying a complex smooth projective surface $S$, they argued that $S$-duality gives rise to a modular transformation of the partition functions. More precisely, let $S$ be a complex smooth projective surface satisfying $H_1(S,\Z) = 0$ and fix a Chern class $c_1 \in H^2(S,\Z)$. Then
\begin{equation} \label{Sdual}
\mathsf{Z}^{\SU(r)}_{c_1}(-1/\tau) = (-1)^{(r-1)\chi(\O_S)} \Big( \frac{ r \tau}{i} \Big)^{-\frac{e(S)}{2}} \mathsf{Z}^{^{L}\SU(r)}_{c_1}(\tau),
\end{equation}
where $e(S)$ denotes the topological Euler characteristic of $S$ and $^{L}\SU(r) = \SU(r) / \Z_r$. For surfaces with $H_1(S,\Z) \neq 0$, some modifications are needed \cite{VW,Wit}.

Vafa-Witten argued that these partition functions are related to generating functions of Euler characteristics of instanton moduli spaces, summing over second Chern class with formal parameter $q = \exp(2 \pi i \tau)$, where $\tau$ is the modular parameter taking values in the upper half plane $\mathfrak{H}$. However, the ``instanton moduli spaces'' required for general surfaces are very singular and the precise algebro-geometric foundations of the theory, on the $\SU(r)$ side, were only recently formulated by Y.~Tanaka and R.P.~Thomas \cite{TT1, TT2}. See also \cite{GSY} for a Donaldson-Thomas perspective.

\subsection*{$\SU(r)$ partition function}  Let $r \in \Z_{>0}$, $c_1 \in H^2(S,\Z)$ algebraic, and $n \in H^4(S,\Z) \cong \Z$. Take a line bundle $L$ on $S$ with $c_1(L) = c_1$.
On the $\SU(r)$ side, the relevant moduli space is
$$
N_S^H(r,L,n) := \big\{ [(E,\phi)] \, : \, \rk(E) = r, \, \det E \cong L, \, c_2(E) = n, \, \tr \phi = 0 \big\},
$$
where $(E,\phi)$ is a stable Higgs pair (with respect to a fixed polarization $H$ on $S$) and $[(E,\phi)]$ denotes its isomorphism class. More precisely, $E$ is a torsion free sheaf and $\phi : E \rightarrow E \otimes K_S$ is a morphism satisfying a Gieseker stability condition for $\phi$-invariant subsheaves of $E$. The gauge group $\SU(r)$ is incorporated by the requirements $\det(E)\cong L$ and $\tr \phi = 0$. Suppose, for the moment, $r$ and $c_1$ are chosen such that there are no rank $r$ strictly semistable Higgs pairs $(E,\phi)$ on $S$ with $c_1(E)=c_1$. Tanaka-Thomas showed that $N_S^H(r,L,n)$ admits a symmetric perfect obstruction theory in the sense of \cite{Beh}. 

In this paper, we only consider surfaces with $H_1(S,\Z) = 0$. Therefore $N_S^H(r,L,n) \cong N_S^H(r,c_1,n)$, where
$$
N_S^H(r,c_1,n) := \big\{[(E,\phi)] \, : \, \rk(E) = r, \, c_1(E) = c_1, \, c_2(E) = n, \, \tr \phi = 0 \big\}
$$ 
and, as before, $[(E,\phi)]$ denotes the isomorphism class of a stable Higgs pair $(E,\phi)$. We will henceforth work with $N_S^H(r,c_1,n)$. Scaling the Higgs field gives a $\C^*$-action on the non-compact moduli space $N_S^H(r,c_1,n)$. However, the fixed locus $N_S^H(r,c_1,n)^{\C^*}$ is projective and Tanaka-Thomas define invariants by the virtual localization formula \cite{GP}
\begin{equation} \label{virloc}
\int_{[N_S^H(r,c_1,n)^{\C^*}]^{\vir}} \frac{1}{e(N^{\vir})} \in \Q,
\end{equation}
where $e$ denotes $\C^*$-equivariant Euler class and $N^{\vir}$ is the virtual normal bundle.\footnote{The equivariant parameter drops out because of symmetry.} The fixed locus $N_S^H(r,c_1,n)^{\C^*}$ contains the Gieseker-Maruyama moduli space $M_S^H(r,c_1,n)$ of rank $r$ stable torsion free sheaves on $S$ with Chern classes $c_1,n$ as an open and closed subscheme. In general, it contains other components sometimes referred to as vertical or monopole components. The virtual dimension of $M_S^H(r,c_1,n)$ equals
\begin{equation} \label{vd}
\vd(r,c_1,n) := 2rn - (r-1)c_1^2 - (r^2-1)\chi(\O_S).
\end{equation}
The $\C^*$-induced perfect obstruction theory on $M_S^H(r,c_1,n)$ has virtual tangent bundle given (point-wise) by $R\Hom_S(E,E)_0[1]$, where $(\cdot)_0$ denotes trace-free part, and was studied by T.~Mochizuki \cite{Moc}. The contribution of $M_S^H(r,c_1,n)$ to \eqref{virloc} equals
$$
(-1)^{\vd(r,c_1,n)} e^{\vir}(M_S^H(r,c_1,n)),
$$
where $e^{\vir}$ denotes the virtual Euler characteristic introduced by Ciocan-Fontanine--Kapranov \cite{CFK} and Fantechi-G\"ottsche \cite{FG}. When $H \cdot K_S < 0$, there are no vertical components, $M_S^H(r,c_1,n)$ is smooth of expected dimension, and its virtual and topological Euler characteristics coincide. In this case, the study of these Euler characteristics has a rich ongoing history (an incomplete selection of literature: \cite{Kly, Got1, Got2, Got3, LQ1, LQ2, Man, Moz, Wei, Yos1, Yos2, Koo} and references therein). 

The mathematical definition of the $\SU(r)$ Vafa-Witten partition function is
$$
\mathsf{Z}^{\SU(r)}_{c_1}(q) := r^{-1} q^{-\frac{1}{2r} \chi(\O_S) + \frac{r}{24} K_S^2} \sum_{n \in \Z} q^{\frac{1}{2r} \vd(r,c_1,n)}(-1)^{\vd(r,c_1,n)}  \int_{[N_S^H(r,c_1,n)^{\C^*}]^{\vir}} \frac{1}{e(N^{\vir})},
$$
where $\vd(r,c_1,n)$ is given by \eqref{vd} and the normalization factor $q^{\bullet}$ is crucial to make modularity work \cite{VW}. We assumed there are no strictly semistable objects. This assumption can be removed by using Joyce-Song pairs \cite{TT2} as we review in Section \ref{sec:JS}. 

A closed expression for the $\SU(2)$ partition function for surfaces with smooth canonical curve was conjectured in \cite{VW} and extended to arbitrary surfaces satisfying $p_g(S)>0$ in \cite{DPS}. This conjectural formula has been verified, up to some order in $q$, in many cases \cite{GK1, Laa1}. For the $\SU(3)$ and $\SU(5)$ partition function, conjectural closed expressions are also known and tested \cite{GK3, Laa1, GKL}. Each of these partition functions is invariant under replacing $c_1$ by $c_1 + r \gamma$ for any algebraic class $\gamma \in H^2(S,\Z)$. 

For $r$ and $c_1$ such that Gieseker and $\mu$-stability coincide (e.g.~when $\gcd(r,H\cdot c_1) = 1$), this is obvious since $- \otimes \O_S(\gamma)$ induces an isomorphism on moduli spaces. 
\begin{conjecture} \label{conj1}
For any smooth polarized surface $(S,H)$ and algebraic classes $c_1, \gamma \in H^2(S,\Z)$, we have $\mathsf{Z}^{\SU(r)}_{c_1}(q) = \mathsf{Z}^{\SU(r)}_{c_1 + r \gamma}(q)$.
\end{conjecture}

\subsection*{$\SU(r) / \Z_r$ partition function} In order to mathematically understand the $S$-duality transformation \eqref{Sdual}, we need to define the right-hand side. In the physics literature \cite{VW, LL}, one can find the following formula
\begin{equation} \label{generalform}
\mathsf{Z}^{\SU(r) / \Z_r}_{c_1}(q) := \sum_{w \in H^2(S,\mu_r)} e^{\frac{2 \pi i}{r}(w \cdot c_1)} \mathsf{Z}_{w}(q),
\end{equation}
where $\mu_r$ is the cyclic group of order $r$ and $\mathsf{Z}_{w}(q)$ should be a generating function of Euler characteristics of moduli spaces of $\mathrm{PGL}(r)$-bundles in class $w$. 
The algebro-geometric compactification of the moduli space of $\mathrm{PGL}(r)$-bundles is the moduli space of twisted sheaves studied in \cite{Lie1, Yos4}. 
The group $H^2(S,\mu_r)$ classifies $\mu_r$-gerbes. In \cite{Jia1}, the first-named author developed Vafa-Witten theory of $\mu_r$-gerbes (see also \cite{Jia2, Jia3, JK, JTs}). 
In this paper, which is inspired by \cite{Jia1}, we take a different approach:
\begin{itemize}
\item Instead of $\mu_r$-gerbes, we use K.~Yoshioka's moduli spaces of twisted sheaves via Brauer-Severi varieties \cite{Yos4}.
\item We define the $\SU(r)/\Z_r$ partition function by appropriately summing over ``Chern character twisted by rational $B$-field'', a notion introduced by D.~Huybrechts and P.~Stellari \cite{HSt}. We sum differently over the Chern data  compared to \cite{Jia1} and in a way which works for arbitrary surfaces.
\end{itemize}

We give a mathematical definition of $\mathsf{Z}_{w}(q)$ for any $w \in H^2(S,\mu_r)$ and then define $\mathsf{Z}_{c_1}^{\SU(r) / \Z_r}(q)$ by \eqref{generalform}. 
We make the following assumption, which leads to striking simplifications and allows us to take various short-cuts.\footnote{On the $\SU(r)$ side, for $r$ prime, Thomas showed that cosection localization \cite[Cor.~5.30]{Tho} leads to surprising simplifications, whereas the non-prime case is currently computationally mostly out of reach.} \\

\noindent \textbf{Assumption.} Let $r$ be a prime number. \\

Consider the short exact sequences
\begin{align} 
\begin{split} \label{ses}
1 &\rightarrow \O_S^* \rightarrow \mathrm{GL}(r) \rightarrow \mathrm{PGL}(r) \rightarrow 1, \\
1 &\rightarrow \mu_r \rightarrow \mathrm{SL}(r) \rightarrow \mathrm{PGL}(r) \rightarrow 1.
\end{split}
\end{align}
From these sequences and the inclusion $\mu_r \leq \O_S^*$, we obtain (commuting) maps 
\begin{align*}
H^1(S,\mathrm{PGL}(r)) \stackrel{\delta}{\rightarrow} H^2(S,\O_S^*), \quad H^1(S,\mathrm{PGL}(r)) \stackrel{\delta'}{\rightarrow} H^2(S,\mu_r), \quad H^2(S,\mu_r) \stackrel{o}{\rightarrow} H^2(S,\O_S^*).
\end{align*}
Fix a class $w \in H^2(S,\mu_r)$. An important consequence of the period-index theorem of A.J.~de Jong is that $\delta'$ is \emph{surjective} \cite[Cor.~4.2.2.4]{Lie2}, \cite{dJo}. 
For $S$ a $K3$ surface, surjectivity of $\delta'$ was proved in \cite{HSc}. Hence we can pick a $\PP^{r-1}$-bundle
$
p : Y \rightarrow S
$
such that $w(Y) := \delta'([Y]) = w$. Our construction will be independent of the choice of $p : Y \rightarrow S$ satisfying $w(Y) = w$. Applying $o$, we obtain a Brauer class
$$
\alpha := o(w) = \delta([Y]) \in H^2(S,\O_S^*)_{\tor},
$$
where $H^2(S,\O_S^*)_{\tor}$ denotes the torsion part of $H^2(S,\O_S^*)$, which is isomorphic to the Brauer group $\mathrm{Br}(S)$ \cite{Gro}.

Let $\Coh(S,\alpha)$ be the category of $\alpha$-twisted sheaves on $S$ \cite{Cal}. Using an analytic trivialization $\{U_i\}_i$ of $p : Y \rightarrow S$, locally $Y_i:=p^{-1}(U_i) \cong \PP^{r-1} \times U_i$ has a tautological line bundle $\O_{Y_i}(\lambda_i)$. These line bundles glue to a $p^* \alpha^{-1}$-twisted line bundle  $L$ on $Y$. Therefore, given any $\alpha$-twisted sheaf $E$ on $S$, we can ``untwist'' it on $Y$ by $p^*E \otimes L$. Coherent sheaves obtained in this way are called $Y$-sheaves \cite{Yos4} and we denote their category by $\Coh(S,Y) \subset \Coh(Y)$. This gives an equivalence between $\Coh(S,Y)$ and $ \Coh(S,\alpha)$,
which extends to an equivalence between the following categories (Section \ref{sec:cat}):
\begin{itemize}
\item $\mathrm{Higgs}_{p^* K_S}(S,Y)$ containing pairs $(E,\phi)$ with $E$ a $Y$-sheaf and $\phi : E \rightarrow E \otimes p^* K_S$. 
\item $\mathrm{Higgs}_{K_S}(S,\alpha)$ consisting of $\alpha$-twisted Higgs pairs on $S$.
\end{itemize}

\subsection*{Trivial Brauer class} When $\alpha = 0$, $\mathrm{Higgs}_{K_S}(S,\alpha)$ is the (untwisted) category of Higgs pairs considered by Tanaka-Thomas \cite{TT1}. By \eqref{ses}, $\alpha = 0$ implies $Y$ is the projectivization of a vector bundle on $S$. Hence there exists an \emph{algebraic} class $\xi \in H^2(S,\Z)$ such that $[\xi] = w$. Virtual counts of (untwisted) Higgs pairs on $S$ were already carried out by Tanaka-Thomas. When there exist no rank $r$ strictly semistable Higgs pairs $(E,\phi)$ on $S$ with $c_1(E)=\xi$, we define as before
\begin{equation} \label{trivialBr}
\mathsf{Z}_{w}(q) := q^{-\frac{1}{2r} \chi(\O_S) + \frac{r}{24} K_S^2} \sum_{n \in \Z} q^{\frac{1}{2r} \vd(r,\xi,n)} (-1)^{\vd(r,\xi,n)} \int_{[N_S^H(r,\xi,n)^{\C^*}]^{\vir}} \frac{1}{e(N^{\vir})},
\end{equation}
where $\vd(r,\xi,n)$ is given by \eqref{vd}. When there are strictly semistable objects, the corresponding invariants are defined via Joyce-Song pairs as discussed in Section \ref{sec:JS}. Conjecture \ref{conj1} implies that $\mathsf{Z}_{w}(q)$ is independent of the choice of algebraic $\xi \in H^2(S,\Z)$ such that $w=[\xi]$. As noted above, this is clear when $\gcd(r, H \cdot w) = 1$.

\subsection*{Non-trivial Brauer class} Suppose $\alpha \neq 0$ and let $\xi \in H^2(S,\Z)$ such that $[\xi] = w$. Then $\xi$ is \emph{non-algebraic}. Let $G$ be the rank $r$ locally free $Y$-sheaf corresponding to the (unique up to scale) extension of $T_{Y/S}$ by $\O_Y$ \cite{Yos4}. We consider the moduli space $M_{Y,\xi/r}^H(r,\xi,n)$ of ($G$-)twisted stable $Y$-sheaves $E$ satisfying
\begin{equation} \label{twistedch}
e^{\frac{\xi}{r}} \ch_G(E) := e^{\frac{\xi}{r}} \frac{\ch(Rp_*(E \otimes G^\vee))}{\sqrt{\ch(Rp_*( G \otimes G^\vee))}} = (r, \xi , \tfrac{1}{2} \xi^2 - n) \in H^*(S,\Q).
\end{equation}
Fixing $e^{\xi/r} \ch_G(E)$ of the $Y$-sheaf $E$ corresponds to fixing Chern character twisted by rational $B$-field $\xi/r$ of the $\alpha$-twisted sheaf $p_*(E \otimes L^\vee)$ \cite{HSt}. Based on an argument of Yoshioka in the $K3$ case, we observe that $n \in \Z$ (Proposition \ref{integrality}) , which is important in our calculations. We also consider the moduli space $N_{Y,\xi/r}^H(r,\xi,n)$ of ($G$-)twisted stable $Y$-Higgs sheaves $(E,\phi)$ with $\tr \phi = 0$ and $E$ satisfying \eqref{twistedch}. Scaling the Higgs field $\phi : E \rightarrow E \otimes p^* K_S$ gives a $\C^*$-action on $N_{Y,\xi/r}^H(r,\xi,n)$.

The $K$-group of $Y$-sheaves is of the form 
$$K(S,Y) = \Z E_0 \oplus K(S,Y)_{\leq 1},$$
where $E_0$ is a locally free $Y$-sheaf of minimal positive rank and $K(S,Y)_{\leq 1}$ is the subgroup generated by $Y$-sheaves $E$ of dimension $\leq 1$ (meaning $\dim p_*E \leq 1$). 
Since $r$ is prime and $\alpha \neq 0$, we have $\rk(E_0)=r$ (Theorem \ref{conethm}).
Hence $M_{Y,\xi/r}^H(r,\xi,n)$ is projective, it contains all isomorphism classes of torsion free $Y$-sheaves $E$ satisfying \eqref{twistedch}, and $N_{Y,\xi/r}^H(r,\xi,n)$ is a \emph{cone} over $M_{Y,\xi/r}^H(r,\xi,n)$ (Theorem \ref{conethm}). So we can drop the $H$-dependence and 
\begin{equation*} 
N_{Y,\xi/r}(r,\xi,n)^{\C^*} = M_{Y,\xi/r}(r,\xi,n).
\end{equation*}
The moduli space  $M_{Y,\xi/r}(r,\xi,n)$ has a natural perfect obstruction theory, which has virtual tangent bundle (point-wise) given by $R\Hom_Y(E,E)_0[1]$ (Proposition \ref{potM}). Its virtual dimension $\vd(r,\xi,n)$ is given by \eqref{vd}. Since $N_{Y,\xi/r}(r,\xi,n)$ is a cone over $M_{Y,\xi/r}(r,\xi,n)$, it carries an induced symmetric perfect obstruction theory by \cite{JTh}. We define
\begin{align*}
\mathsf{Z}_{w}(q) :=&\, q^{-\frac{1}{2r} \chi(\O_S) + \frac{r}{24} K_S^2} \sum_{n \in \Z} q^{\frac{1}{2r} \vd(r,\xi,n)} (-1)^{\vd(r,\xi,n)} \int_{[N_{Y,\xi/r}(r,\xi,n)^{\C^*}]^{\vir}} \frac{1}{e(N^{\vir})} \\
 =&\, q^{-\frac{1}{2r} \chi(\O_S) + \frac{r}{24} K_S^2} \sum_{n \in \Z} q^{\frac{1}{2r} \vd(r,\xi,n)} e^{\vir}(M_{Y,\xi/r}(r,\xi,n)).
\end{align*}
The second equality, and independence of choice of $[p : Y \rightarrow S] \in H^1(S,\mathrm{PGL}(r))$ such that $w(Y)=w$ and $\xi \in H^2(S,\Z)$ such that $[\xi] = w$, are shown in Proposition \ref{indep}.  It also follows that $\mathsf{Z}_{w}(q)$ is independent of $H$.

We have defined $\mathsf{Z}_{w}(q)$ for all $w \in H^2(S,\mu_r)$ and therefore the $\SU(r) / \Z_r$ Vafa-Witten partition function by \eqref{generalform}. We can now state Vafa-Witten's $S$-duality transformation as a mathematical conjecture.

\begin{conjecture} \label{Sdualconj}
Let $(S,H)$ be a smooth polarized  surface with $H_1(S,\Z) = 0$ and $p_g(S)>0$. Let $r$ be prime and $c_1 \in H^2(S,\Z)$ algebraic. Then $\sfZ_{c_1}^{\SU(r)}(q)$ and $\sfZ_{c_1}^{\SU(r) / \Z_r}(q)$ are Fourier expansions in $q = \exp(2 \pi i \tau)$ of meromorphic functions $\sfZ_{c_1}^{\SU(r)}(\tau)$ and $\sfZ_{c_1}^{\SU(r) / \Z_r}(\tau)$ on $\mathfrak{H}$ satisfying
\begin{equation*} 
\mathsf{Z}^{\SU(r)}_{c_1}(-1/\tau) = (-1)^{(r-1)\chi(\O_S)} \Big( \frac{ r \tau}{i} \Big)^{-\frac{e(S)}{2}} \mathsf{Z}^{\SU(r) / \Z_r}_{c_1}(\tau).
\end{equation*}
\end{conjecture}

For surfaces with $p_g(S) = 0$, the partition functions are modular in a more complicated sense; e.g.~for $S = \PP^2$ and $r=2$ they are mock modular forms of weight $-3/2$ \cite{VW, Kly, Man}. Our interest is in surfaces with $p_g(S)>0$. 

\begin{remark} \label{GK3}
In \cite{GK3, GKL}, the authors took the following approach to the $S$-duality transformation \eqref{Sdual}. For $\SU(2)$, $\SU(3)$, and $\SU(5)$, the conjectural closed formulae for the right hand side of \eqref{trivialBr} in \cite{VW,DPS, GK3, GKL} are explicit formulae involving modular forms and $\chi(\O_S), K_S^2, a_ia_j, K_S \xi, a_i\xi, \xi^2$, where $a_i$ runs over the Seiberg-Witten basic classes of $S$. These formulae satisfy Conjecture \ref{conj1} and make sense for \emph{non-algebraic} classes $\xi$, thereby providing an ad hoc definition of $\mathsf{Z}^{\SU(r) / \Z_r}_{c_1}(q)$. See \cite{GK3, GKL} for the precise statement. A non-trivial calculation shows that this ad hoc definition satisfies $S$-duality (\cite{VW,DPS} for the $\SU(2)$ case, \cite{GK3} for the $\SU(3)$ case, and \cite{GKL} for the $\SU(5)$ case). 
\end{remark}

\subsection*{$K3$ surfaces}

For $S$ a $K3$ surface, $\mathsf{Z}_{c_1}^{\SU(r)}(q)$ was determined by Tanaka-Thomas \cite{TT2} using a multiple cover formula of Y.~Toda \cite{Tod1, Tod2} proved by Maulik-Thomas \cite{MT}. This requires equating virtual invariants to invariants defined by Behrend function. The fact that ``virtual'' and ``motivic'' invariants coincide is a special feature of $K3$ surfaces and  generally does not hold for other surfaces satisfying $p_g(S)>0$ \cite{MT}. The formula for $\mathsf{Z}_{c_1}^{\SU(r)}(q)$ satisfies Conjecture \ref{conj1}. 

We carry out the calculation on the $\SU(r) / \Z_r$-side. For $w \in H^2(S,\mu_r)$ with trivial Brauer class, the contribution to the partition function immediately reduces to Tanaka-Thomas's calculation. For $w \in H^2(S,\mu_r)$ with non-trivial Brauer class, there are no vertical contributions and we use Yoshioka's deformation of moduli spaces of twisted sheaves to Hilbert schemes of points. In order to describe our result, we define 
\begin{align}
\begin{split} \label{prep}
\delta_{ab} &:= \Bigg\{\begin{array}{cc} 1 & \mathrm{if} \ a-b \in rH^2(S,\Z) \\ 0 & \mathrm{otherwise,} \end{array} \\
\Delta(q) &:= q \prod_{n=1}^{\infty} (1-q^n)^{24}.
\end{split}
\end{align}

\begin{theorem} \label{mainthm}
For any $K3$ surface $S$, prime number $r$, generic\footnote{The precise genericity we (and \cite{TT2}) require is described in Section \ref{sec:pfmainthm}.} polarizations, and $c_1 \in H^2(S,\Z)$ algebraic, we have 
\begin{align*}
\mathsf{Z}^{\SU(r) / \Z_r}_{c_1}(q) &= \sum_{w \in H^2(S,\mu_r)} e^{\frac{2 \pi i}{r}(w \cdot c_1)} \mathsf{Z}_{w}(q), \quad \textrm{where} \\
\sfZ_{w}(q) &= \frac{\delta_{w,0}}{r^2} \, \Delta(q^r)^{-1} + \frac{1}{r} \sum_{j=0}^{r-1} e^{-\frac{\pi i j}{r} w^2} \,  \Delta(e^{\frac{2 \pi i j}{r}} q^{\frac{1}{r}})^{-1}.
\end{align*}
\end{theorem}

Combining Theorem \ref{mainthm} with Tanaka-Thomas's formula for $\mathsf{Z}^{\SU(r)}_{c_1}(q)$ and some interesting lattice theoretic identities for $H^2(S,\Z)$, called flux sums in physics (Section \ref{sec:Sdual}), we establish the $S$-duality conjecture for $K3$ surfaces and prime rank:

\begin{theorem} \label{maincor}
The $S$-duality conjecture (Conjecture \ref{Sdualconj}) holds for all $K3$ surfaces, prime ranks, and generic polarizations.
\end{theorem}

\noindent \textbf{Acknowledgements.} We thank Amin Gholampour, Lothar G\"ottsche, and Ties Laarakker for helpful discussions related to this paper. Special thanks go to Richard Thomas, who suggested looking at Brauer classes during discussions on Remark \ref{GK3}. 
The authors would like to thank the Institute of Mathematical Sciences at ShanghaiTech, where most of this work was carried out. Y.J.~is partially supported by NSF DMS-1600997. M.K.~is supported by NWO grant VI.Vidi.192.012.

\section{Moduli space}

\subsection{Categories of twisted sheaves} \label{sec:cat}

Let $S$ be a smooth projective surface satisfying $H_1(S,\Z) = 0$ and let $r \in \Z_{>0}$. For the moment, we do not assume $r$ is prime. We discuss two important equivalences of categories. 

\subsection*{Spectral construction} The first is the spectral construction.

Let $\pi : X=\mathrm{Tot}_S(K_S) \rightarrow S$ be the total space of the canonical bundle on $S$. Push-forward along $\pi$ induces an equivalence between the category of quasi-coherent sheaves on $X$ and the category of quasi-coherent sheaves of $\Sym^\mdot K_S^{\vee}$-algebras on $S$. The latter can also be seen as pairs $(E,\phi)$, where $E$ is a quasi-coherent sheaf on $S$ and $\phi : E \rightarrow E \otimes K_S$ is a morphism. Denote by $\mathrm{Coh}_c(X)$ the category of coherent sheaves on $X$ with proper support. We write $\mathrm{Higgs}_{K_S}(S)$ for the category of pairs $(E,\phi)$ as above with $E$ coherent, which we refer to as Higgs pairs on $S$. Push-forward along $\pi$ yields an equivalence 
$$
\pi_* : \mathrm{Coh}_c(X) \stackrel{\sim}{\rightarrow} \mathrm{Higgs}_{K_S}(S), \quad \mathcal{E} \mapsto \pi_* \cE. 
$$

Let $p : Y \rightarrow S$ be a $\PP^{r-1}$-bundle corresponding to a class in $[Y] \in H^1(S,\mathrm{PGL}(r))$. Using the maps from the introduction, let $w(Y) := \delta'([Y]) \in H^2(S,\mu_r)$ and let $\alpha := o(w) \in H^2(S,\O_S^*)$. Consider an analytic open cover $\{U_i\}_i$ on which $Y$ trivializes, i.e.~$Y_i := p^{-1}(U_i) \cong U_i \times \PP^{r-1}$, and denote by $\O_{Y_i}(\lambda_i)$ a tautological line bundle on $Y_i$. 

The class $\alpha$ can be represented by a \v{C}ech cocycle $\{\alpha_{ijk} \in \Gamma(U_i \cap U_j \cap U_k,\O_S^*)\}_{i,j,k}$. An $\alpha$-twisted sheaf on $S$ is a collection $\{(E_i,\chi_{ij})\}_{i,j}$, where $E_i$ is a coherent sheaf on $U_i$ and $\chi_{ij} : E_i|_{U_i \cap U_j} \rightarrow E_j|_{U_i \cap U_j}$ is an isomorphism, such that
$$
\chi_{ii} = \mathrm{id}, \quad \chi_{ji} = \chi_{ij}^{-1}, \quad \chi_{ki} \circ \chi_{jk} \circ \chi_{ij} = \alpha_{ijk} \cdot \mathrm{id}, \quad \forall i,j,k.
$$ 
We denote the category of $\alpha$-twisted sheaves on $S$ by $\mathrm{Coh}(S,\alpha)$. (Derived) categories of $\alpha$-twisted sheaves were introduced by A.~C$\breve{\textrm{a}}$ld$\breve{\textrm{a}}$raru \cite{Cal}. 

An $\alpha$-twisted Higgs pair on $S$ is a pair $(E,\phi)$, where $E$ is an $\alpha$-twisted sheaf on $S$ and $\phi : E \rightarrow E \otimes K_S$ is a morphism. We denote the category of $\alpha$-twisted Higgs pairs on $S$ by $\mathrm{Higgs}_{K_S}(S,\alpha)$. Representing $\pi^* \alpha$ by $\{ \pi^* \alpha_{ijk} \in \Gamma(\pi^{-1}(U_i \cap U_j \cap U_k),\O_X^*)\}_{i,j,k}$, the spectral construction gives an equivalence
$$
\pi_* : \mathrm{Coh}_c(X,\pi^* \alpha) \stackrel{\sim}{\rightarrow} \mathrm{Higgs}_{K_S}(S,\alpha), \quad \mathcal{E} \mapsto \pi_* \cE.
$$

\subsection*{$Y$-sheaves} We recall Yoshioka's notion of a $Y$-sheaf \cite[Def.~1.3]{Yos4}. A coherent sheaf $E$ on $Y$ is called a $Y$-sheaf if there exists an isomorphism
$$
E|_{Y_i} \cong p^*(E_i) \otimes \O_{Y_i}(\lambda_i)
$$
for some coherent sheaf $E_i$ on $U_i$ for all $i$. Following Yoshioka's notation, we denote by $\Coh(S,Y) \subset \Coh(Y)$ the category of $Y$-sheaves. By \cite[Lem.~1.5]{Yos4},  ``$E \in \Coh(S,Y)$'' is an open condition on $\Coh(Y)$. We fix a system of isomorphisms
$
\eta_{ij} : \O_{Y_i \cap Y_j}(\lambda_i) \rightarrow \O_{Y_i \cap Y_j}(\lambda_j),
$
where $\eta_{ii} = \mathrm{id}$, $\eta_{ji} = \eta_{ij}^{-1}$, and $\eta_{ki} \circ \eta_{jk} \circ \eta_{ij} = p^*\alpha_{ijk}^{-1} \cdot \mathrm{id}$ for all $i,j,k$. Then $\{(\O_{Y_i}(\lambda_i),\eta_{ij})\}_{i,j}$ defines a $p^* \alpha^{-1}$-twisted line bundle on $Y$ denotes by $L(p^* \alpha^{-1})$. There is an equivalence of categories
$$
\Lambda^{L(p^* \alpha^{-1})} : \Coh(S,Y) \stackrel{\sim}{\rightarrow} \Coh(S,\alpha), \quad E \mapsto p_*(E \otimes L(p^* \alpha^{-1})^{\vee}). 
$$

We combine both constructions. Let $X_Y := X \times_S Y$. We denote the projections by $\pi : X_Y \rightarrow Y$ and $p : X_Y \rightarrow X$. At the same time, $X_Y = \mathrm{Tot}_Y(p^* K_S)$ and $X_Y$ is a $\PP^{r-1}$-bundle over $X$ trivialized by $\{\pi^{-1}(U_i)\}_i$:
\begin{align} 
\begin{split} \label{Cartesiansq}
\xymatrix
{
X_Y \ar^p[r] \ar_{\pi}[d] & X \ar^{\pi}[d] \\
Y \ar_p[r] & S.
}
\end{split}
\end{align}
Using the $p^* \pi^*\alpha^{-1}$-twisted line bundle $L(p^* \pi^*(\alpha^{-1}))$ on $X_Y$, we obtain an equivalence between the categories of $X_Y$-sheaves and $\pi^* \alpha$-twisted sheaves on $X$. Moreover, the spectral construction applied to $\pi : X_Y \rightarrow Y$ gives an equivalence
$$
\pi_* : \mathrm{Coh}_c(X_Y) \stackrel{\sim}{\rightarrow} \mathrm{Higgs}_{p^*K_S}(Y), \quad \cE \mapsto \pi_* \cE,
$$
where $\mathrm{Higgs}_{p^*K_S}(Y)$ denotes the category of pairs $(E,\phi)$ with $E$ a coherent sheaf on $Y$ and $\phi : E \rightarrow E \otimes p^* K_S$ a morphism. Moreover, if $\cE$ is an $X_Y$-sheaf then $E = \pi_* \cE$ is a $Y$-sheaf (by diagram \eqref{Cartesiansq} and the projection formula). Denoting by $\mathrm{Higgs}_{p^* K_S}(S,Y)$ the category of pairs $(E,\phi)$ with $E$ a $Y$-sheaf, we obtain a diagram 
\begin{displaymath}
\xymatrix@C=3cm
{
\mathrm{Coh}_c(X,X_Y) \ar_{\sim}^{\Lambda^{L(p^* \pi^*\alpha^{-1})}}[r] \ar^{\sim}_{\pi_*}[d] &\mathrm{Coh}_c(X,\pi^*\alpha) \ar_{\sim}^{\pi_*}[d] \\
\mathrm{Higgs}_{p^* K_S}(S,Y) \ar^{\sim}_{\Lambda^{L(p^* \alpha^{-1})}}[r] & \mathrm{Higgs}_{K_S}(S,\alpha).
}
\end{displaymath}
Commutativity of the diagram follows from diagram \eqref{Cartesiansq},  $L(p^* \pi^*\alpha^{-1}) \cong \pi^* L(p^*\alpha^{-1})$, and the projection formula. We refer to elements of $\mathrm{Higgs}_{p^* K_S}(S,Y)$ as $Y$-Higgs pairs.

\subsection{Twisted Chern character} \label{sec:chern}

As in the previous section, $S$ is a smooth projective surface with $H_1(S,\Z)=0$, $r \in \Z_{>0}$ (not necessarily prime), and $p : Y \rightarrow S$ is a $\PP^{r-1}$-bundle. Let $w:=w(Y) \in H^2(S,\mu_r)$. We are interested in Chern characters of $Y$-sheaves.  

Denote the Grothendieck group of $Y$-sheaves by $K(S,Y)$. The dimension of a $Y$-sheaf $E$ is defined as $\dim p_* E$. The following result is \cite[Lem.~3.2]{Yos4}.\footnote{In loc.~cit.~this lemma is stated in a section on $K3$ surfaces, but the proof holds for any $S$.}
\begin{lemma}\cite[Lem.~3.2]{Yos4} \label{Yoslem1}
\hfill
\begin{enumerate}
\item There exists a locally free $Y$-sheaf $E_0$ such that $$\rk(E_0) = \min \{ \rk E > 0 \, : \, E \in \Coh(S,Y) \}.$$ 
\item  $K(S,Y) = \Z E_0 \oplus K(S,Y)_{\leq 1}$, where $K(S,Y)_{\leq 1}$ denotes the subgroup generated by $Y$-sheaves of dimension $\leq 1$.
\end{enumerate}
\end{lemma}

Since we assume $H_1(S,\Z) = 0$, the groups $H_*(S,\Z)$, $H^*(S,\Z)$ do not have torsion and, by the universal coefficient theorem, we have isomorphisms $H^k(S,\Z)\otimes \mu_s \cong H^k(S,\mu_s)$ for all $k,s \geq 0$. By Leray--Hirsch, there exists a monic polynomial $f(x) \in H^*(S,\Z)[x]$ such that \cite[Lem.~1.6]{Yos4}
$$
H^*(Y,\Z) \cong H^*(S,\Z)[x] / (f(x)),
$$
where $\deg x = 2$. In particular, $p^*$ gives an inclusion $H^*(S,\Z)  \hookrightarrow H^*(Y,\Z) $, which remains injective after applying $- \otimes \mu_s$ for any $s \geq 0$.
For any $Y$-sheaf $E$ of rank $s$, $[c_1(E) \mod s] \in H^2(Y,\mu_s)$ lies in the image of $p^*$ and one defines \cite[Def.~1.4]{Yos4}
$$
w(E) := (p^*)^{-1} [c_1(E) \mod s] \in H^2(S,\mu_s).
$$
We also need the following result \cite[Lem.~3.1]{Yos4}.\footnote{In \cite{Yos4} this lemma is stated in a section on $K3$ surfaces but it holds for any $S$ with $H_1(S,\Z)_{\mathrm{tor}} = 0$.}
\begin{lemma}\cite[Lem.~3.1]{Yos4} \label{Yoslem2}
For any $Y$-sheaf of rank $\rk(E)$, we have\footnote{Using the cup product on $H^2(S,\Z)$, $[(s-1) v^2 \mod 2s] \in \Z / 2s\Z$ is well-defined for any $v \in H^2(S,\mu_s)$ and $s > 0$.}
$$
c_2(Rp_*(E \otimes E^\vee)) \equiv -(\rk(E)-1) w(E)^2 \mod 2\rk(E).
$$
\end{lemma}

Denote by $T_{Y/S}$ the relative tangent bundle of $p : Y \rightarrow S$, then \cite[Lem.~1.1]{Yos4}
$$
\Ext^1_Y(T_{Y/S},\O_Y) = \C.
$$
We denote by $G$ the sheaf determined by the unique (up to scale) non-trivial extension. Continuing the notation of the previous section, $G$ can also be obtained by gluing
$$
G_i := G|_{Y_i} \cong p^*( (p_* \O_{Y_i}(\lambda_i))^\vee  )(\lambda_i).
$$
This also shows that $G$ is a locally free $Y$-sheaf of rank $r$.

\begin{definition} \label{defchG}
For any $Y$-sheaf $E$, we define
$$
\ch_{G}(E) := \frac{\ch(Rp_*(E \otimes G^\vee))}{\sqrt{\ch(Rp_*( G \otimes G^\vee))}} \in H^*(S,\mathbb{Q}).
$$
\end{definition}
 
In this definition, writing $\ch_{G}(E) = (s, \zeta, \tfrac{1}{2} \zeta^2 - b) \in H^*(S,\mathbb{Q})$, we have $s = \rk(E)$ and the classes $\zeta, b$ are in general \emph{rational}. The following proposition states that, they become integer after ``twisting with appropriate rational $B$-field''. For $K3$ surfaces and using Mukai vector $v_G(E)$ instead of $\ch_G(E)$, the following proposition reduces to \cite[Lem.~3.3]{Yos4}. We closely follow Yoshioka's proof. 
\begin{proposition} \label{integrality}
Let $\xi \in H^2(S,\Z)$ such that $[\xi] = w \in H^2(S,\mu_r)$. Let $E$ be a $Y$-sheaf of rank $\rk(E)$ and write $e^{\xi / r}\ch_{G}(E) = (s, D, \tfrac{1}{2} D^2 - n) \in H^*(S,\mathbb{Q})$. Then
$$
s = \rk(E) \in \Z, \quad D \in H^2(S,\Z), \quad n \in \Z.
$$
\end{proposition}
\begin{proof} 
Denote $\ch_{G}(E) = (s, \zeta, \tfrac{1}{2} \zeta^2 - b) \in H^*(S,\mathbb{Q})$. We already noted that $s = \rk(E)$. For integrality of $D$, it suffices to show $p^* D \in H^2(Y,\Z)$. Since $E$ is a $Y$-sheaf, we have $\ch(p^* Rp_* ( E \otimes G^\vee)) = \ch(E \otimes G^\vee)$ and therefore
$$
p^* D = p^*(\zeta + \tfrac{s}{r} \xi) = c_1(E) - \tfrac{s}{r} c_1(G) + \tfrac{s}{r} p^* \xi = c_1(E) - s \Big( \frac{c_1(G) - p^* \xi}{r}\Big).
$$
Moreover, $[\xi] = w$, which implies $p^*[\xi]= p^* w = [c_1(G) \mod r] \in H^2(Y,\mu_r)$ by \cite[Lem.~1.3]{Yos4}. Hence the right hand side is an element of $H^2(Y,\Z)$. We also deduce $[p^* D \mod s] = [c_1(E) \mod s]$, i.e.~$[D \mod s] = w(E)$, which we use below.

Next, we show $n \in \Z$ by calculating $\chi(E,E) := \sum_i (-1)^i \dim \Ext_Y^i(E,E)$ in two ways. Since $Y$ is a $\mathbb{P}^{r-1}$-bundle, we have $\chi(\O_Y) = \chi(\O_S)$. 
By Hirzebruch- and Grothendieck-Riemann-Roch
\begin{align*}
\chi(\O_Y)-\chi(E,E) &= \chi(\O_S)-\int_Y \ch(E \otimes E^\vee) \td_Y\\
&= \chi(\O_S)-\int_S \ch(Rp_*(E \otimes E^\vee)) \td_S\\
&= c_2(Rp_*(E \otimes E^\vee)) - (s^2-1) \chi(\O_S) \\
&\equiv -(s-1)w(E)^2- (s^2-1) \chi(\O_S) \mod 2s,
\end{align*} 
where the third equality uses $c_1(Rp_*(E \otimes E^\vee)) = 0$ and the fourth equality follows from Lemma \ref{Yoslem2}. Next, we use the following crucial identity satisfied by any $Y$-sheaves $E_1,E_2$
$$
Rp_*(E_1 \otimes E_2^\vee) \otimes Rp_*(G \otimes G^\vee) \cong Rp_*(E_1 \otimes G^\vee) \otimes (Rp_*(E_2 \otimes G^\vee))^\vee.
$$
Using Definition \ref{defchG}, we deduce 
\begin{align*}
\chi(\O_Y)-\chi(E,E) &= \chi(\O_S)-\int_S \ch(Rp_*(E \otimes E^\vee)) \td_S \\ 
&= \chi(\O_S) - \int_S \ch_{G}(E)  \ch_{G}(E)^\vee \td_S \\
&=\chi(\O_S) - \int_S e^{\frac{\xi}{r}}\ch_{G}(E)  (e^{\frac{\xi}{r}}\ch_{G}(E))^\vee \td_S \\
&= 2 s n - (s-1) D^2 - (s^2-1) \chi(\O_S).
\end{align*} 
Since $[D \mod s] = w(E)$, we have $(s-1)D^2 \equiv (s-1)w(E)^2 \mod 2s$. Combining both expressions for $\chi(E,E)$, we conclude $2 s n \equiv 0 \mod 2s$. Therefore $n \in \Z$ as desired.
\end{proof}

\begin{remark}
Consider the setting of Proposition \ref{integrality} and let $B:=\xi/r$. For $\alpha := o(w)$, let $F:=p_*(E \otimes L(p^* \alpha^{-1})^\vee)$ be the $\alpha$-twisted sheaf corresponding to $E$. As explained by Huybrechts-Stellari \cite{HSt}, from $F$ and $B$ one can construct an untwisted sheaf $F_{B}$ on $S$ and define $\ch^B(F):=\ch(F_B)$, where $\ch^B$ is known as ``Chern character twisted by rational $B$-field''. Then $e^{\xi / r}\ch_{G}(E) = \ch^B(F)$ \cite[Rem.~3.2]{Yos4}. 
\end{remark}

The ``integrality'' stated in Proposition \ref{integrality} is a key ingredient for our definition of the $\SU(r) / \Z_r$ Vafa-Witten partition function. Just like on the $\SU(r)$ side, it allows us to sum over all $n \in \Z$.

\subsection{Moduli and obstruction theory}

\subsection*{Moduli of twisted sheaves} We recall Yoshioka's moduli space of twisted sheaves \cite{Yos4}. See also \cite{Lie1} in the language of gerbes. Let $(S,H)$ be a smooth polarized surface satisfying $H_1(S,\Z)=0$, $r \in \Z_{>0}$ (not necessarily prime), and let $p : Y \rightarrow S$ be a $\PP^{r-1}$-bundle. Let $w := w(Y) \in H^2(S,\mu_r)$ and $\xi \in H^2(S,\Z)$ such that $[\xi] = w$. We denote by $G$ be the sheaf corresponding to the unique (up to scale) non-trivial extension of $T_{Y/S}$ by $\O_Y$. 

For a $Y$-sheaf $E$, its $G$-twisted Hilbert polynomial is defined by 
$$
\chi(G,E \otimes p^* \O(mH) = \chi(p_*(E \otimes G^\vee)(m)).
$$
If $E$ is 2-dimensional, i.e.~$\dim p_* E = 2$, this polynomial has degree two and we denote its leading term by $\tfrac{1}{2} a_2^G(E) m^2$. Then $E$ is called $G$-twisted stable (with respect to $H$) when
$$
\frac{\chi(p_*(F \otimes G^\vee)(m))}{a_2^G(F)} < \frac{\chi(p_*(E \otimes G^\vee)(m))}{a_2^G(E)}
$$
for all $Y$-subsheaves $0 \neq F \subsetneq E$. For $D \in H^2(S,\Z)$ and $n \in \Z$, we denote by
$$
M_{Y,\xi/r}^{H}(s,D,n)
$$
the moduli space of $G$-twisted stable torsion free $Y$-sheaves $E$ satisfying
\begin{equation} \label{fixchB}
e^{\frac{\xi}{r}} \ch_G(E) = (s,D,\tfrac{1}{2} D^2 - n),
\end{equation}
where $s = \rk(E)$, $D \in H^2(S,\Z)$, and $n \in \Z$ by Proposition \ref{integrality}.
We use the subscript $\xi/r$ in $M_{Y,\xi/r}^{H}(s,D,n)$ to stress that we use the rational $B$-field $\xi/r$ for fixing Chern data \eqref{fixchB}. As in the case of (ordinary) moduli of stable sheaves, $M_{Y,\xi/r}^{H}(s,D,n)$ exists as a coarse moduli space and is quasi-projective \cite[Thm.~2.1]{Yos4}. 

\begin{remark}
The notion of $G$-twisted stability can be defined with respect to any locally free $Y$-sheaf $G$. However, this does not lead to new moduli spaces by \cite[Rem.~2.2]{Yos4}. Since we do not vary $G$, we will simply refer to ``$G$-twisted stability'' as ``twisted stability'' (and similarly in analogous settings below).
\end{remark} 

Let $M:=M_{Y,\xi/r}^{H}(s,D,n)$ and let $\pi : Y \times M \rightarrow M$ denote the projection. Although a universal sheaf $\cE$ may not exist globally on $Y \times M$, the complex 
$$
R \hom_{\pi}(\cE,\cE) := R\pi_* R\hom(\cE,\cE) 
$$
exists globally on $Y \times M$ \cite{Cal, HL}. Denote by $\mathbb{L}_M = \tau^{\geq -1} L_M$ the truncated cotangent complex of $M$.
The following is well-known in the untwisted case:
\begin{proposition} \label{potM}
For $s \in \Z_{>0}$, the moduli space $M:=M_{Y,\xi/r}^{H}(s,D,n)$ has a perfect  obstruction theory 
$E^\mdot :=(R \hom_{\pi}(\cE,\cE)_0[1])^\vee \rightarrow \mathbb{L}_M$
of virtual dimension
$$
\vd(s,D,n) := 2sn - (s-1)D^2 - (s^2-1)\chi(\O_S).
$$
\end{proposition}
\begin{proof}
Huybrechts-Thomas \cite{HT} constructed a truncated Atiyah class $$A(\cE) \in \Ext^1_{Y \times M}(\cE,\cE \otimes \mathbb{L}_{Y \times M}),$$ where $\mathbb{L}_{Y \times M} = \LL_Y \boxplus \LL_M$. As in \cite{HT}, $A(\cE)$ provides an obstruction theory via adjunctions
$$
E^\mdot :=(R \hom_{\pi}(\cE,\cE)_0[1])^\vee \rightarrow \mathbb{L}_M.
$$

Now $E^\mdot$ can be represented by a finite complex of vector bundles ($M$ is quasi-projective). Using the equivalence $\mathrm{Coh}(S,Y) \cong \mathrm{Coh}(S,\alpha)$ (Section \ref{sec:cat}), Serre duality for twisted sheaves on $S$, and the projection formula, we obtain
\begin{align*}
\Ext_Y^i(E,F) &\cong \Ext_S^i(p_*(E \otimes L(p^* \alpha^{-1})^\vee),p_*(F \otimes L(p^* \alpha^{-1})^\vee)) \\
&\cong \Ext_S^{2-i}(p_*(F \otimes L(p^* \alpha^{-1})^\vee),p_*(E \otimes L(p^* \alpha^{-1})^\vee) \otimes K_S)^* \\
&\cong \Ext_S^{2-i}(p_*(F \otimes L(p^* \alpha^{-1})^\vee),p_*(E \otimes p^* K_S \otimes L(p^* \alpha^{-1})^\vee))^* \\
&\cong \Ext^{2-i}_Y(F,E \otimes p^* K_S)^\vee, 
\end{align*}
for all $Y$-sheaves $E,F$ and $i \in \Z$. Since $\Hom_Y(E,E)_0 = 0$ (stable sheaves are simple), we find that $\Ext_Y^i(E,E)_0=0$ unless $i=1,2$. Hence $E^\mdot$ can be represented by a 2-term complex of vector bundles \cite[Lem.~4.2]{HT}. The virtual dimension equals $\chi(\O_Y)-\chi(E,E)$, which was already calculated in the proof of Proposition \ref{integrality}. 
\end{proof}

\subsection*{Moduli of twisted Higgs pairs} 

As in Section \ref{sec:cat}, let $X:= \mathrm{Tot}_S(K_S)$ and $X_Y := X \times_S Y$. We consider pure 2-dimensional $X_Y$-sheaves $\cE$ with proper support. Then $\cE$ is twisted stable (with respect to $H$) when
$$
\frac{\chi(p_*(\F \otimes \pi^*G^\vee)(m))}{a_2^{\pi^*G}(\F)} < \frac{\chi(p_*(\cE \otimes \pi^*G^\vee)(m))}{a_2^{\pi^*G}(\cE)}
$$
for all $Y$-subsheaves $0 \neq \F \subsetneq \cE$. We denote by
$
\widetilde{N}_{Y,\xi/r}^{H}(s,D,n)
$
the moduli space of twisted stable pure 2-dimensional $X_Y$-sheaves $\cE$ with proper support satisfying
$$
e^{\frac{\xi}{r}} \ch_G(\pi_* \cE) = (s,D,\tfrac{1}{2} D^2 - n).
$$
By \cite[Thm.~2.1]{Yos4}, $\widetilde{N}_{Y,\xi/r}^{H}(s,D,n)$ exists as a coarse moduli space. 

Recall from Section \ref{sec:cat} the equivalence 
\begin{equation} \label{equivforstab}
\pi_* : \Coh_c(X,X_Y) \stackrel{\sim}{\rightarrow} \mathrm{Higgs}_{p^*K_S}(S,Y).
\end{equation}
\begin{definition}
A $Y$-Higgs pair $(E,\phi)$ is called twisted stable when
$$
\frac{\chi(p_*(F \otimes G^\vee)(m))}{a_2^G(F)} < \frac{\chi(p_*(E \otimes G^\vee)(m))}{a_2^G(E)}
$$
for all $\phi$-invariant $Y$-subsheaves $0 \neq F \subsetneq E$.
\end{definition}

As in \cite[Lem.~2.9]{TT1}, the two notions of stability coincide, so we can view $\widetilde{N}_{Y,\xi/r}^{H}(s,D,n)$ as a moduli space of twisted stable $Y$-Higgs pairs.
\begin{lemma}
Let $\cE$ be an $X_Y$-sheaf with proper support corresponding to $(E,\phi)$ under the equivalence \eqref{equivforstab}. Then $\cE$ is twisted stable if and only if $(E,\phi)$ is twisted stable.
\end{lemma}
\begin{proof}
By diagram \eqref{Cartesiansq} and the projection formula. 
\end{proof}
As in the introduction, we are interested in the moduli $N_{Y,\xi/r}^{H}(s,D,n) \subset \widetilde{N}_{Y,\xi/r}^{H}(s,D,n)$ of twisted stable $Y$-Higgs pairs $(E,\phi)$ with $\tr(\phi) = 0$.

\subsection*{Cone construction} 

Although one could now painstakingly try to develop the analog of \cite{TT1, TT2} for moduli of stable $Y$-sheaves (similar to \cite{Jia1,Jia2,Jia3,JK,JTs} in the language of $\mu_r$-gerbes), we here take a short-cut when \emph{$r$ is prime}. Using the cone construction of Jiang-Thomas \cite{JTh}, we will see that $N_{Y,\xi/r}^{H}(r,D,n)$ has a natural symmetric perfect obstruction theory.
 
 Let $M$ be a quasi-projective scheme. For any coherent sheaf $F$ on $M$, consider the (abelian) cone 
 $$
\Pi :  C(F) := \Spec \Sym^\mdot F \rightarrow M.
 $$
 The grading on $\Sym^\mdot$ gives a $\C^*$-action on $C(F)$ with fixed locus $C(F)^{\C^*} = M$. Suppose
$$
E^\mdot = \{ E^{-1} \rightarrow E^0 \} \rightarrow \mathbb{L}_M
$$ 
is a perfect obstruction theory on $M$ and consider the obstruction sheaf $\mathrm{Ob}:=h^1((E^\mdot)^\vee)$. It induces a tautological relative perfect obstruction theory by \cite[Lem.~2.1]{JTh}
$$
\Pi^* (E^\mdot)^\vee[1] \rightarrow \LL_{C(\mathrm{Ob})/M}.
$$
We want to make it absolute by finding the dotted arrows in the following diagram
\begin{displaymath}
\xymatrix
{
\Pi^* (E^\mdot) \ar@{-->}[r] \ar[d] & F^\mdot \ar@{-->}[r] \ar@{-->}[d] & \Pi^* (E^\mdot)^\vee[1] \ar[d] \\
\Pi^* \LL_{M} \ar[r] & \LL_{C(\mathrm{Ob})} \ar[r] & \LL_{C(\mathrm{Ob})/M}.
}
\end{displaymath}
Jiang-Thomas proved that the dotted arrows exist when $M$ is the classical truncation of a quasi-smooth derived scheme $\boldsymbol{M}$.\footnote{Then $N$ is the classical truncation of the $(-1)$-shifted cotangent bundle $T^*_{\boldsymbol{M}}[-1]$.} On a more basic level, Jiang-Thomas point out that the dotted arrows exist when the entire diagram is restricted to the $\C^*$-fixed locus, in which case one can take $F^\mdot = E^\mdot \oplus (E^\mdot)^\vee[1]$. Hence there exists a commutative diagram
\begin{align} 
\begin{split} \label{diagonfxloc}
\xymatrix
{
E^\mdot \ar[r] \ar[d] & E^\mdot \oplus (E^\mdot)^\vee[1] \ar[r] \ar[d] & (E^\mdot)^\vee[1] \ar[d] \\
\LL_{M} \ar[r] & \LL_{C(\mathrm{Ob})}|_M \ar[r] & \LL_{C(\mathrm{Ob})/M}|_M.
}
\end{split}
\end{align}

Take $M:=M_{Y,\xi/r}^{H}(s,D,n)$ with the perfect obstruction theory from Proposition \ref{potM}. Then $\mathrm{Ob} = \ext^2_{\pi}(\cE,\cE)_0$. Note that for any $[E] \in M$, the fibre of $\Pi : C(\mathrm{Ob}) \rightarrow M$ is
$$
(\mathrm{Ob}|_{[E]})^* = \Ext_Y^2(E,E)_0^* \cong \Hom_Y(E,E \otimes p^* K_S)_0
$$
by Serre duality for $Y$-sheaves (see proof of Proposition \ref{potM}). Therefore, the closed points of $C(\mathrm{Ob})$ correspond to isomorphism classes of trace-free $Y$-Higgs pairs $(E,\phi)$ with $[E] \in M$ and we obtain an open subset
\begin{equation} \label{incl}
C(\mathrm{Ob}) \subset N:=N_{Y,\xi/r}^{H}(s,D,n).
\end{equation}
This observation was made in \cite{JTh} when $M$ is the ordinary moduli space of stable sheaves on $S$. On the $\SU(r)$-side, this is not very useful for surfaces satisfying $p_g(S)>0$, because (in general) $N \setminus C(\mathrm{Ob}) \neq \varnothing$. On the $\SU(r) / \Z_r$-side, this is extremely useful by the following theorem.
\begin{theorem} \label{conethm}
Suppose $w \in H^2(S,\mu_r)$ has non-trivial Brauer class and $r$ is prime. Let $M:=M_{Y,\xi/r}^{H}(r,D,n)$, $N:=N_{Y,\xi/r}^{H}(r,D,n)$. Then $M$ is projective, it contains all isomorphism classes of torsion free $Y$-sheaves $E$ with $e^{\xi/r} \ch_G(E) = (r,D,\tfrac{1}{2}D^2-n)$, and $N = C(\mathrm{Ob})$, where $\mathrm{Ob}$ is the obstruction sheaf from Proposition \ref{potM}. In particular, there exists a morphism $F^\mdot \cong E^\mdot \oplus (E^\mdot)^\vee[1] \rightarrow \mathbb{L}_{N}|_M$, which fits in commutative diagram  \eqref{diagonfxloc}.
\end{theorem}
\begin{proof}
Consider the minimal rank $\rk(E_0)$ in Proposition \ref{Yoslem1}. Since there exists a locally free $Y$-sheaf of rank $r$ (namely $G$ in Section \ref{sec:chern}), we have $\rk(E_0) | r$. Therefore $\rk(E_0)=1$ or $\rk(E_0) = r$, because $r$ is prime. If $\rk(E_0) = 1$, then $\alpha = 0$ contrary to our assumption \cite[Rem.~2.2]{Yos4}. Hence $\rk(E_0) = r$. 
(Alternatively, by the period-index theorem \cite[Thm.~4.2.2.3]{Lie2} \cite{dJo}, the order $\mathrm{ord}(\alpha) = r$ (period of $\alpha$) equals the minimal rank of a locally free $\alpha$-twisted sheaf on $S$ (index of $\alpha$), therefore $\rk(E_0)=r$.)
Consequently, any rank $r$ torsion free $Y$-sheaf has no non-trivial saturated $Y$-subsheaves and is automatically twisted stable (with respect to any polarization $H$). In particular, there are no rank $r$ twisted strictly semistable $Y$-sheaves, so $M$ is projective and \eqref{incl} is an equality. By \cite{JTh}, we obtain diagram \eqref{diagonfxloc}.
\end{proof}

\begin{remark}
Moduli spaces of stable sheaves on a smooth projective variety are classical truncations of derived schemes \cite{PTVV}. For $r$ prime, $w \in H^2(S,\mu_r)$ with non-trivial Brauer class, it follows that $M_{Y,\xi/r}^{H}(r,D,n)$ is also the classical truncation of a derived scheme, because the property of a rank $r$ torsion free sheaf on $Y$ to be a $Y$-sheaf is \emph{open} \cite[Lem.~1.5]{Yos4}. Therefore Theorem \ref{conethm} can be strengthened to say that $N_{Y,\xi/r}^{H}(r,D,n)$ has a natural symmetric perfect obstruction theory by \cite{JTh}.  For the definition of invariants in the next section this is irrelevant: since $N$ is non-compact, invariants are defined by a virtual localization formula on the fixed locus, so Theorem \ref{conethm} suffices. 
\end{remark}

\section{Virtual invariants}

\subsection{$\SU(r)/\Z_r$ partition function}

We are now ready to define the $\SU(r)/\Z_r$ Vafa-Witten partition function. Fix a smooth polarized surface $(S,H)$ with $H_1(S,\Z) = 0$. Let $r$ be a prime number and let $c_1 \in H^2(S,\Z)$ be an algebraic class. The partition function will depend on $(S,H)$, $r$, $c_1$. We suppress the dependence on $(S,H)$ from the notation. 

For any $w \in H^2(S,\mu_r)$, we will define $\sfZ_{w}(q)$ and then we set
\begin{equation*} 
\mathsf{Z}^{\SU(r) / \Z_r}_{c_1}(q) := \sum_{w \in H^2(S,\mu_r)} e^{\frac{2 \pi i}{r}(w \cdot c_1)} \mathsf{Z}_{w}(q).
\end{equation*}
Let $\alpha := o(w) \in H^2(S,\O_S^*)_{\mathrm{tor}} \cong \mathrm{Br}(S)$ \cite{Gro}. 

\subsection*{Trivial Brauer class} Suppose $\alpha =0$. Then there exists an algebraic $\xi \in H^2(S,\Z)$ such that $[\xi] = w$. Since $\mathrm{Higgs}_{K_S}(S,\alpha) \cong \mathrm{Higgs}_{K_S}(S)$, as mentioned in the introduction, one can use Tanaka-Thomas theory to count stable Higgs pairs. Suppose there exist no rank $r$ strictly semistable Higgs pairs $(E,\phi)$ on $S$ with $c_1(E)=\xi$, then we use the symmetric perfect obstruction theory of \cite{TT1} and define
\begin{equation*} 
\mathsf{Z}_{w}(q) := q^{-\frac{1}{2r} \chi(\O_S) + \frac{r}{24} K_S^2} \sum_{n \in \Z} q^{\frac{1}{2r} \vd(r,\xi,n)} (-1)^{\vd(r,\xi,n)} \int_{[N_S^H(r,\xi,n)^{\C^*}]^{\vir}} \frac{1}{e(N^{\vir})},
\end{equation*}
where $\vd(r,\xi,n)$ is given in \eqref{vd}. When there are strictly semistable objects, we instead use Joyce-Song pairs as in \cite{TT2}. This is discussed in detail in Section \ref{sec:JS}. For $w \neq 0$, we have $\gcd(r,\xi)=1$. For each $n$ one could choose a generic polarization $H$ w.r.t.~$(r,\xi,n)$ (i.e.~not lying on a wall determined by $(r,\xi,n)$), then there are no rank $r$ strictly $H$-semistable Higgs pairs $(E,\phi)$ on $S$ with $c_1(E) = \xi$ and $c_2(E) = n$. 
For $w=0$, strictly semistable objects are unavoidable.

Above, we chose an algebraic lift $\xi \in H^2(S,\Z)$ of $w \in H^2(S,\mu_r)$. For $\gcd(r,H\cdot w)=1$, the operation $- \otimes \O_S(\gamma)$ with $\gamma \in H^2(S,\Z)$ algebraic, is an isomorphism on Higgs moduli spaces preserving the symmetric perfect obstruction theories. Then the definition of $\mathsf{Z}_{w}(q)$ does not depend on choice of algebraic lift $\xi$.\footnote{On the components corresponding to eigenrank $(1, \ldots, 1)$ independence upon replacing $\xi$ by $\xi + r\gamma$ is known \cite{Laa1, Laa2}.}
In general, independence of choice of algebraic $\xi \in H^2(S,\Z)$ satisfying $[\xi]=w$ is an open question and is equivalent to Conjecture \ref{conj1}. It is known for $K3$ surfaces as we discuss in Section \ref{sec:Sdual}.

\subsection*{Non-trivial Brauer class} 

Suppose $\alpha \neq 0$. Choose a $\PP^{r-1}$-bundle $p : Y \rightarrow S$ such that $w(Y) = w$. The fact that this is possible is \emph{non-trivial} and follows from surjectivity of the map $\delta' : H^1(S,\mathrm{PGL}(r)) \rightarrow H^2(S,\mu_r)$ discussed in the introduction. This surjectivity is a consequence of the period-index theorem \cite[Cor.~4.2.2.4]{Lie2} \cite{dJo}. Choose $\xi \in H^2(S,\Z)$ such that $[\xi] = w$. Then $\xi$ is \emph{non-algebraic}. By Theorem \ref{conethm}, we have
$$
N_{Y,\xi/r}(r,\xi,n) = C(M_{Y,\xi/r}(r,\xi,n)), \quad N_{Y,\xi/r}(r,\xi,n)^{\C^*} = M_{Y,\xi/r}(r,\xi,n).
$$
In the proof of Theorem \ref{conethm}, we saw that all rank $r$ torsion free $Y$-sheaves are twisted stable for any polarization $H$, so we dropped $H$ from the notation. Using the $\C^*$-fixed perfect obstruction theory and virtual normal bundle from Theorem \ref{conethm}, we define
\begin{align*}
\mathsf{Z}_{w}(q) := q^{-\frac{1}{2r} \chi(\O_S) + \frac{r}{24} K_S^2} \sum_{n \in \Z} q^{\frac{1}{2r} \vd(r,\xi,n)} (-1)^{\vd(r,\xi,n)} \int_{[N_{Y,\xi/r}(r,\xi,n)^{\C^*}]^{\vir}} \frac{1}{e(N^{\vir})}.
\end{align*}
For this definition we made two choices: $p : Y \rightarrow S$ such that $[w(Y)] = w$ and $\xi \in H^2(S,\Z)$ such that $[\xi]=w$. We show independence of the choices.
\begin{proposition} \label{indep}
For $r$ prime and $w \in H^2(S,\mu_r)$ with non-trivial Brauer class, we have
\begin{equation} \label{evirgenfun}
\mathsf{Z}_{w}(q) = q^{-\frac{1}{2r} \chi(\O_S) + \frac{r}{24} K_S^2} \sum_{n \in \Z} q^{\frac{1}{2r} \vd(r,\xi,n)} e^{\vir}(M_{Y,\xi/r}(r,\xi,n)).
\end{equation}
Moreover, the right hand side is independent of the choice of $\PP^{r-1}$-bundle $p : Y \rightarrow S$ satisfying $w(Y) = w$ and $\xi \in H^2(S,\Z)$ such that $[\xi] = w$. 
\end{proposition}
\begin{proof}
Equation \eqref{evirgenfun} follows from \cite[Prop.~3.3]{JTh}. This can be seen as follows. Denote by $y$ the weight one representation of the trivial $\C^*$-action on $M:=M_{Y,\xi/r}(r,\xi,n)$ and let $t:=c_1^{\C^*}(y)$. Since $F^\mdot \cong E^\mdot \oplus (E^\mdot)^\vee \otimes y^{-1} [1]$ ($\C^*$-equivariantly), we have
$$
\frac{1}{e(N^{\vir})} = e(E^\mdot \otimes y) = t^{\vd(M)} c_{-\frac{1}{t}}((E^\mdot)^{\vee}),
$$
where $c_{x}(\cdot) = 1 + x c_1(\cdot) + \cdots$ and $(E^\mdot)^{\vee} = T_M^{\vir}$ is the virtual tangent bundle of $M$.

Suppose $p : Y \rightarrow S$, $p' : Y' \rightarrow S$ are two $\PP^{r-1}$-bundles such that $w(Y) = w(Y') = w$.  
Consider the projections $p_1 : Y' \times_S Y \rightarrow Y'$ and $p_2 : Y' \times_S Y \rightarrow Y$. By \cite[Lem.~1.7]{Yos4}, there exists a line bundle $L$ on $Y' \times_S Y$ such that
$$
\Xi : \Coh(S,Y) \stackrel{\sim}{\rightarrow} \Coh(S,Y'), \quad E \mapsto p_{1*}(p_2^* E \otimes L) 
$$ 
is an equivalence of categories.\footnote{$L$ is of the form (suppressing pull-backs) $L(p^{\prime *} \alpha^{-1}) \otimes L(p^* \alpha^{-1})^{\vee} \otimes P$ with $P \in \Pic(S)$.} Let $G$, $G'$ be locally free sheaves corresponding to the non-trivial extensions $T_{Y/S}$ by $\O_{Y}$ and $T_{Y'/S}$ by $\O_{Y'}$ respectively. Since $\Xi(G)$ is a rank $r$ torsion free $Y'$-sheaf, we have 
$$
\Xi(G) = G' + T \in K(S,Y'),
$$
where $T$ is a $Y'$-sheaf of dimension zero (i.e.~$\dim(p_*T)=0$). As in \cite[Lem.~3.5]{Yos4}, we therefore find
$$
\ch_{G}(E) = \ch_{\Xi(G)}(\Xi(E)) = \ch_{G'}(\Xi(E))
$$
and $\Xi$ induces an isomorphism of moduli spaces $M_{Y,\xi/r}(r,\xi,n) \cong M_{Y',\xi/r}(r,\xi,n)$. 

Finally let $p : Y \rightarrow S$ be a $\PP^{r-1}$-bundle satisfying $w(Y)=w$ and let $\xi, \xi' \in H^2(S,\Z)$ be two lifts of $w$. Then $\xi' = \xi+r\gamma$ for some $\gamma \in H^2(S,\Z)$. A $Y$-sheaf $E$ satisfies
$$
e^{\frac{\xi'}{r}} \ch_G(E) = (r,\xi',\tfrac{1}{2} \xi^{\prime 2} - n')
$$
if and only if
$$
e^{\frac{\xi}{r}} \ch_G(E) = (r,\xi,\tfrac{1}{2} \xi^{2} - n), \quad n' =n + (r-1) \gamma \xi +\tfrac{1}{2} r (r-1) \gamma^2.
$$
In particular, $M_{Y,\xi'/r}(r,\xi',n') = M_{Y,\xi/r}(r,\xi,n)$ and $\vd(r,\xi',n') = \vd(r,\xi,n)$. Consequently the right hand side of \eqref{evirgenfun} is independent of the choice of lift $\xi$ of $w$.
\end{proof}

\begin{remark}
We defined the generating function by fixing a polarization $H$ in advance. However, for fixed $r,\xi$, it could be convenient to choose a different polarization for each $n$, e.g.~(as mentioned above) a generic polarization with respect to $(r,\xi,n)$.
\end{remark}

\begin{remark}
For surfaces satisfying $p_g(S)>0$, conjecturally the generating functions $\sfZ_{c_1}^{\SU(r)}(q)$ and $\sfZ_{c_1}^{\SU(r) / \Z_r}(q)$ are independent of $H$. On the $\SU(r)$ side, for  $r$ prime and in the case there are no strictly semistable objects, this independence follows from Mochizuki's formula \cite{Moc, GK1} and Laarakker's formula \cite{Laa1}. 
On the $\SU(r) / \Z_r$ side, for $w \in H^2(S,\mu_r)$ with non-trivial Brauer class, independence of $\sfZ_{w}(q)$ of polarization follows from the proof of Theorem \ref{conethm}, where we saw that any rank $r$ torsion free $Y$-sheaf has no non-trivial saturated $Y$-subsheaves and is therefore automatically twisted stable with respect to any polarization $H$.
\end{remark}

\subsection{Joyce-Song pairs} \label{sec:JS}

This section is a review of Tanaka-Thomas's handling of strictly semistable objects on the $\SU(r)$-side. On the $\SU(r) / \Z_r$-side, for $r$ prime, strictly semistable objects only occur for the contribution of $w \in H^2(S,\mu_r)$ with trivial Brauer class. Moreover, for generic polarizations, strictly semistable objects only occur for $w=0$. 

Let $(S,H)$ be a smooth polarized surface satisfying $H_1(S,\Z) = 0$, $p_g(S)>0$, and $r \in \Z_{>0}$ (not necessarily prime).
Fix a Chern character
$$
\ch := (r,c_1,\tfrac{1}{2} c_1^2 - n).
$$
For any integer $\nu \gg 0$, consider the moduli space of stable Joyce-Song Higgs pairs
$$
P_{S,\nu}^H(r,c_1,n) := \Big\{[(E,\phi,s)] \, : \, \rk(E) = r, \, c_1(E) = c_1, \, c_2(E) = n, \, \tr \phi = 0, \, 0 \neq s \in H^0(E(\nu H))  \Big\},
$$ 
where a triple $(E,\phi,s)$ is called stable if and only if $(E,\phi)$ is a semistable Higgs pair and if $s$ factors through a $\phi$-invariant subsheaf $0 \neq F \subsetneq E$, then we have 
$
p_F(m) < p_E(m),
$
where $p_F(m), p_E(m)$ are the reduced Hilbert polynomials of $F,E$. There is no notion of strictly semistable Joyce-Song Higgs pair. The integer $\nu \gg 0$ can be taken such that $H^{>0}(S,E(\nu H)) = 0$ for any $[(E,\phi,s)] \in P:=P_{S,\nu}^H(r,c_1,n)$. In particular, $\dim H^0(S,E(\nu H)) = \chi(\ch(\nu H)) - 1$, where
$$
\chi(\ch(\nu H)) := \int_S \ch \cdot e^{\nu H} \cdot \td_S.
$$

The moduli space $P$ is a quasi-projective scheme. It can be seen as a moduli space of complexes $I^\mdot = \{\O_X(-\nu \pi^* H) \rightarrow \cE \}$, where $\cE$ is a 2-dimensional sheaf with proper support on $X = \mathrm{Tot}_S(K_S)$ and $\pi : X \to S$ denotes the projection. Following the work of D.~Joyce and Y.~Song \cite{JS}, this is used in  \cite{TT2} to show that $P$ has a symmetric perfect obstruction theory. Scaling the Higgs field gives a $\C^*$-action on $P$ with projective fixed locus and, as before, one can consider invariants
$$
\int_{[P^{\C^*}]^{\vir}} \frac{1}{e(N^{\vir})}, 
$$
which now depends on the choice of $\nu \gg 0$. 

For any Chern character $\ch = (r,c_1,\tfrac{1}{2} c_1^2 - n)$ with $r>0$, we write $p_{\ch}(m) = \chi(\ch(m H)) / r$ for the reduced Hilbert polynomial determined by $\ch$ with respect to the polarization $H$.
\begin{definition} \label{TTgeneral}
 A polarization $H$ on $P:=P_{S,\nu}^H(r,c_1,n)$ is called generic in the sense of \cite[Eqn.~(2.4)]{TT2} when for any Chern character $\ch':=(r',c_1', \tfrac{1}{2} c_1^{\prime 2} - n')$ with $r'>0$ we have
$$
p_{\ch}(m) = p_{\ch'}(m) \Longrightarrow  (r,c_1,n) = d \cdot (r',c_1',n')
$$
for some $d \in \Z_{\geq 1}$.
\end{definition}

The following was conjectured in \cite{TT2} and proved on the vertical component by T.~Laarakker \cite{Laa2}.
\begin{conjecture}[Tanaka-Thomas] \label{conj2} 
Let $H$ be generic in the sense of \cite[Eqn.~(2.4)]{TT2} and let $P:=P_{S,\nu}^H(r,c_1,n)$. Define $\VW_S^{H}(r,c_1,n)$ by the equation
$$
\int_{[P^{\C^*}]^{\vir}} \frac{1}{e(N^{\vir})} = (-1)^{\chi(\ch(\nu H))-1} \chi(\ch(\nu H)) \cdot  \VW_S^{H}(r,c_1,n).
$$
Then $\VW_S^{H}(r,c_1,n)$ is independent of the choice of $\nu \gg 0$. 
\end{conjecture}

When there are no rank $r$ strictly semistable Higgs pairs $(E,\phi)$ on $S$ with $c_1(E) = c_1$ and $c_2(E) = n$ we have \cite[Thm.~1.5]{TT2}
$$
\VW_S^{H}(r,c_1,n) = \int_{[N^{\C^*}]^{\vir}} \frac{1}{e(N^{\vir})},
$$ 
where $N:=N_S^H(r,c_1,n)$. In the definition of the $\SU(r)$ partition function, whenever there are strictly semistable objects for $(r,c_1,n)$, the corresponding integral in the partition function is replaced by $\VW_S^{H}(r,c_1,n)$. 

Let $r$ be prime and $w \in H^2(S,\mu_r)$ with trivial Brauer class. Let $\xi \in H^2(S,\Z)$ algebraic such that $[\xi] = w$. In the definition of the contribution of $\sfZ_{w}(q)$ to the $\SU(r)/\Z_r$ partition function, whenever there are strictly semistable objects for $(r,\xi,n)$, the corresponding integral in $\sfZ_{w}(q)$ is replaced by $\VW_S^{H}(r,\xi,n)$.

\subsection{Partition functions for $K3$ surfaces} \label{sec:pfmainthm}

We turn to the proof of Theorem \ref{mainthm}. Let $S$ be a $K3$ surface and denote by $\Hilb^n(S)$ the Hilbert scheme of $n$ points on $S$. 
The following result, due to G\"ottsche \cite{Got1}, determines their topological Euler characteristics
\begin{align} 
\begin{split} \label{Hilbformulae}
\sum_{n=0}^{\infty} e(\Hilb^n(S)) \, q^{n-1} &= \Delta(q)^{-1}. 
\end{split}
\end{align}
The following lemma will be useful in this section.
\begin{lemma} \label{lem}
Let $\psi(x) = \sum_{n=0}^{\infty} \psi_n x^n \in \C[[x]]$ be a formal power series.
Then for any prime number $r$ and $k \in \Z$, we have
$$
\sum_{n \equiv k \mod r} \psi_n x^n = \frac{1}{r} \sum_{j=0}^{r-1} e^{-\frac{2 \pi i jk}{r}} \psi(e^{\frac{2 \pi i j}{r}} x),
$$
where the sum is over all $n \in \Z_{\geq 0}$ satisfying $n \equiv k \mod r$ and $i := \sqrt{-1}$.
\end{lemma}
\begin{proof}
This follows from 
$$
\frac{1}{r} \sum_{j=0}^{r-1} e^{-\frac{2 \pi i jk}{r}} \psi(e^{\frac{2 \pi i j}{r}} x) = \sum_{n=0}^{\infty}  \Big[  \frac{1}{r} \sum_{j=0}^{r-1} e^{\frac{2 \pi i}{r} j(n-k)} \Big] \psi_n x^n,
$$
where the sum between brackets is 1 if and only if $n \equiv k \mod r$ and 0 otherwise.
\end{proof}

Let $S$ be a $K3$ surface, $r$ prime, and $c_1 \in H^2(S,\Z)$ algebraic. The $\SU(r)$ partition function was determined by Tanaka-Thomas 
\begin{equation} \label{SUK3}
\sfZ_{c_1}^{\SU(r)}(q) = \frac{\delta_{c_1,0}}{r^3} \, \Delta(q^r)^{-1} + \frac{1}{r^2} \sum_{j=0}^{r-1} e^{-\frac{\pi i j}{r} c_1^2} \,  \Delta(e^{\frac{2 \pi i j}{r}} q^{\frac{1}{r}})^{-1},
\end{equation}
where the polarizations are chosen generic in the sense we now explain. Note that \eqref{SUK3} satisfies Conjecture \ref{conj1}.

For $\gcd(r,c_1)=1$, for each $n$ we choose $H$ generic with respect to $(r,c_1,n)$, meaning that $H$ does not lie on a wall determined by $(r,c_1,n)$. Then there are no rank $r$ strictly semistable Higgs pairs $(E,\phi)$ on $S$ with $c_1(E) = c_1$ and $c_2(E)=n$. By Thomas's cosection theorem \cite[Thm.~5.34]{Tho}, the only component of $N_S^H(r,c_1,n)^{\C^*}$ contributing to the Vafa-Witten invariant is $M_S^H(r,c_1,n)$. Since $M_S^H(r,c_1,n)$ is smooth of expected dimension, its contribution to the Vafa-Witten invariant is \cite[Sect.~7.1]{TT1}
$$
e^{\vir}(M_S^H(r,c_1,n)) = e(M_S^H(r,c_1,n)).
$$ 
By \cite{OG,Huy,Yos3}, $M_S^H(r,c_1,n)$ is deformation equivalent to $\Hilb^{\vd(r,c_1,n)/2}(S)$, where $\vd(r,c_1,n)$ is given by \eqref{vd}. Hence \eqref{SUK3} follows at once from \eqref{Hilbformulae} combined with Lemma \ref{lem}. The case $\gcd(r,c_1)=r$ is much harder. Choosing $H$ generic in the sense of \cite[Eqn.~2.4]{TT2} (Definition \ref{TTgeneral}), \eqref{SUK3} was proved in \cite{TT2} by comparing virtual and motivic invariants and using a multiple cover formula established in \cite{MT}.

Let $w \in H^2(S,\mu_r)$. Suppose $w$ has trivial Brauer class. Let $\xi \in H^2(S,\Z)$ be algebraic such that $[\xi] = w$. Then (by definition) $\mathsf{Z}_{w}(q) = r \sfZ_{\xi}^{\SU(r)}(q)$, which is given (for generic polarizations as above) by \eqref{SUK3}. Clearly $\mathsf{Z}_{w}(q)$ is independent of the choice of $\xi \in H^2(S,\Z)$ algebraic such that $[\xi] = w$. The proof of Theorem \ref{mainthm} follows from the following result, where $H$ can be \emph{any} polarization.

\begin{proposition}
For any polarized $K3$ surface $(S,H)$, prime number $r$, and $w \in H^2(S,\mu_r)$ with non-trivial Brauer class, we have
\begin{align*}
\sfZ_{w}(q) =  \frac{1}{r} \sum_{j=0}^{r-1} e^{-\frac{\pi i j}{r} w^2} \,  \Delta(e^{\frac{2 \pi i j}{r}} q^{\frac{1}{r}})^{-1}.
\end{align*}
\end{proposition}
\begin{proof}
Let $p : Y \rightarrow S$ be a $\PP^{r-1}$-bundle such that $w(Y) = w$ and $\xi \in H^2(S,\Z)$ such that $[\xi] = w$. By Theorem \ref{conethm}, $M_{Y,\xi/r}(r,\xi,n)$ contains all isomorphism classes of torsion free $Y$-sheaves $E$ satisfying $e^{\xi/r}\ch_G(E) = (r,\xi,\tfrac{1}{2} \xi^2-n)$ and twisted stability is automatically satisfied. Since $S$ is a $K3$ surface, we have $\Ext^2_Y(E,E)_0 \cong \Hom_Y(E,E)_0^* = 0$ (stable sheaves are simple). Therefore $M_{Y,\xi/r}(r,\xi,n)$ is smooth of expected dimension $\vd(r,\xi,n) = 2rn - (r-1)\xi^2 - 2(r^2-1)$. Moreover, in \cite[Thm.~3.16]{Yos4} Yoshioka proved that $M_{Y,\xi/r}(r,\xi,n)$ is deformation equivalent to  
$
\Hilb^{\vd(r,\xi,n)/2}(S).
$
(Recall that $\beta^2 \in 2 \Z$ for any $\beta  \in H^2(S,\Z)$, so indeed $\vd(r,\xi,n) \in 2\Z$.) From Proposition \ref{indep}, we deduce
\begin{align*}
\mathsf{Z}_{w}(q) &= q^{-\frac{1}{r} } \sum_{n \in \Z} q^{\frac{1}{2r} \vd(r,\xi,n)} e(M_{Y,\xi/r}(r,\xi,n)) \\
&= q^{-\frac{1}{r} } \sum_{n \in \Z} q^{\frac{1}{2r} \vd(r,\xi,n)} e(\Hilb^{\vd(r,\xi,n)/2}(S)) \\
&=  q^{-\frac{1}{r}} \cdot \frac{1}{r} \sum_{j=0}^{r-1} e^{\frac{2 \pi i j}{r} \{\frac{1}{2}(r-1)\xi^2  + r^2-1\}} \Big\{\sum_{n \in \Z} e(\Hilb^n(S)) \, Q^n \Big\} \Big|_{Q = e^{\frac{2 \pi i j}{r}} q^{\frac{1}{r}}},
\end{align*}
where we used Lemma \ref{lem} for the last equality. The result follows from \eqref{Hilbformulae}.
 \end{proof}

\subsection{$S$-duality for $K3$ surfaces} \label{sec:Sdual}

We now prove the $S$-duality conjecture (Conjecture \ref{Sdualconj}) for any $K3$ surface $S$, prime rank $r$, generic polarizations, and $c_1 \in H^2(S,\Z)$ algebraic.

\begin{proof}[Proof of Theorem \ref{maincor}]
Recall that $\Delta(q)$ given in \eqref{prep} is the Fourier expansion in $q = e^{2 \pi i \tau}$ of the discriminant modular form $\Delta(\tau)$ on $\mathfrak{H}$, which has weight 12.  For $K3$ surfaces, the $S$-duality transformation states
\begin{equation} \label{K3Sdual}
\sfZ^{\SU(r)}_{c_1}(-1/\tau) = r^{-11} \, \tau^{-12} \,  \sfZ^{\SU(r) / \Z_r}_{c_1}(\tau),
\end{equation}
where $\sfZ^{\SU(r)}_{c_1}(\tau)$ and $\sfZ^{\SU(r) / \Z_r}_{c_1}(\tau)$ stand for the meromorphic functions determined respectively by \eqref{SUK3} and Theorem \ref{mainthm}. 

We require the following two identities for the $K3$ lattice known as flux sums in physics \cite{VW, LL}. For any $j \in \{1, \ldots, r-1\}$, we define $n_j \in \{1, \ldots, r-1\}$ by the equation $jn_j \equiv -1 \mod r$. Then
\begin{align*}
\sum_{w \in H^2(S,\mu_r)} e^{\frac{2 \pi i}{r} (wc_1)} = r^{22} \delta_{c_1,0}, \quad \sum_{w \in H^2(S,\mu_r)} e^{\frac{2 \pi i}{r} (wc_1)} e^{-\frac{\pi i j}{r} w^2} = r^{11} e^{-\frac{\pi i n_j}{r} c_1^2}.
\end{align*}
Proofs of these identities can be found e.g.~in \cite[Sect.~4.2]{GK3}. Combining with 
\begin{align*}
\Delta(r\tau)|_{-\frac{1}{\tau}} &= \big( \tfrac{\tau}{r} \big)^{12} \,  \Delta\big(\tfrac{\tau}{r} \big), \\
\Delta\big( \tfrac{\tau+j}{r} \big)\big|_{-\frac{1}{\tau}} &= \tau^{12} \Delta\big( \tfrac{\tau+n_j}{r} \big), \quad \forall j=1,\ldots, r-1,
\end{align*}
transformation \eqref{K3Sdual} follows.
\end{proof}


\begin{thebibliography}{PTVV}
\bibitem[Beh]{Beh} K.~Behrend, \textit{Donaldson-Thomas type invariants via microlocal geometry}, Annals of Math.~170 (2009) 1307--1338.
\bibitem[Cal]{Cal} A.~C$\breve{\textrm{a}}$ld$\breve{\textrm{a}}$raru, \textit{Derived categories of twisted sheaves on Calabi-Yau manifolds}, PhD thesis Cornell University (2000).
\bibitem[CFK]{CFK} I.~Ciocan-Fontanine and M.~Kapranov, \textit{Virtual fundamental classes via dg-manifolds}, Geom.~Topol.~13 (2009) 1779--1804.
\bibitem[DPS]{DPS} R.~Dijkgraaf, J.-S.~Park, and B.~J.~Schroers, \textit{$N=4$ supersymmetric Yang-Mills theory on a K\"ahler surface}, hep-th/9801066 ITFA-97-09.
\bibitem[FG]{FG} B.~Fantechi and L.~G\"ottsche, \textit{Riemann-Roch theorems and elliptic genus for virtually smooth schemes}, Geom.~Topol.~14 (2010) 83--115.
\bibitem[GSY]{GSY} A.~Gholampour, A.~Sheshmani, and S.-T.~Yau, \textit{Localized Donaldson-Thomas theory of surfaces}, Amer.~Jour.~Math.~142 (2020) 405--442.
\bibitem[Got1]{Got1} L.~G\"ottsche, \textit{The Betti numbers of the Hilbert scheme of points on a smooth projective surface}, Math.~Ann.~286 (1990) 193--207. 
\bibitem[Got2]{Got2} L.~G\"ottsche, \textit{Change of polarization and Hodge numbers of moduli spaces of torsion free sheaves on surfaces}, Math.~Z.~223 (1996) 247--260.
\bibitem[Got3]{Got3} L.~G\"ottsche, \textit{Theta functions and Hodge numbers of moduli spaces of sheaves on rational surfaces}, Comm.~Math.~Phys.~206 (1999) 105--136.
\bibitem[GK1]{GK1} L.~G\"ottsche and M.~Kool, \textit{Virtual refinements of the Vafa-Witten formula}, Commun.~Math.~Phys.~376 (2020) 1--49.
\bibitem[GK3]{GK3} L.~G\"ottsche and M.~Kool, \textit{Refined $\mathrm{SU}(3)$ Vafa-Witten invariants and modularity}, Pure and Appl.~Math.~Quart.~14 (2018) 467--513.
\bibitem[GKL]{GKL} L.~G\"ottsche, M.~Kool, and T.~Laarakker, \textit{$\mathrm{SU}(r)$ Vafa-Witten invariants, Ramanujan's continued fractions, and cosmic strings},  	arXiv:2108.13413.
\bibitem[GP]{GP} T.~Graber and R.~Pandharipande, \textit{Localization of virtual classes}, Invent.~Math.~135 (1999) 487--518. 
\bibitem[Gro]{Gro} A.~Grothendieck, \textit{Le groupe de Brauer}, in: J.~Giraud (ed) et al.: Dix expos\'es sur la cohomologie des sch\'emas, 46--189, North-Holland, Amsterdam (1968).
\bibitem[Huy]{Huy}  D.~Huybrechts,  \textit{Compact hyper-K\"ahler manifolds: basic results},  Invent.~Math.~135 (1999) 63--113.
\bibitem[HL]{HL} D.~Huybrechts, M.~Lehn, \textit{The geometry of moduli spaces of sheaves}, Cambridge Univ.~Press (2010).
\bibitem[HSc]{HSc} D.~Huybrechts and S.~Schr\"oer, \textit{The Brauer group of analytic $K3$ surfaces}, IMRN 50 (2003) 2687--2698. 
\bibitem[HSt]{HSt} D.~Huybrechts and P.~Stellari, \textit{Equivalences of twisted $K3$ surfaces}, Math.~Ann.~332 (2005) 901--936. 
\bibitem[HT]{HT} D.~Huybrechts and R.~P.~Thomas, \textit{Deformation-obstruction theory for complexes via Atiyah and Kodaira-Spencer classes}, Math.~Ann.~346 (2010) 545--569. 
\bibitem[Jia1]{Jia1} Y.~Jiang, \textit{Counting twisted sheaves and S-duality}, arXiv:1909.04241.
 \bibitem[Jia2]{Jia2} Y.~Jiang, \textit{The Vafa-Witten invariants for surface Deligne-Mumford stacks and S-duality}, Survey for ICCM-2019, arXiv:1909.03067.
 \bibitem[Jia3]{Jia3} Y.~Jiang, \textit{On the construction of moduli stack of projective Higgs bundles over surfaces}, arXiv:1911.00250.
\bibitem[JK]{JK} Y.~ Jiang  and  P.~Kundu,  \textit{The  Tanaka-Thomas's  Vafa-Witten  invariants  for  surface  Deligne-Mumford stacks}, Pure and Appl.~Math.~Quart.~17 (2021) 503--573.
\bibitem[JTh]{JTh} Y.~Jiang and R.~P.~Thomas, \textit{Virtual signed Euler characteristics}, J.~Alg.~Geom.~26 (2017) 379--397.
\bibitem[JTs]{JTs} Y.~Jiang and H.-H.~Tseng, \textit{Stable pair invariants for K3 gerbes and higher rank S-duality conjecture for $K3$ surfaces}, arXiv:2003.09562.
\bibitem[dJo]{dJo} A.~J.~de Jong, \textit{The period-index problem for the Brauer group of an algebraic surface}, Duke Math.~Jour.~123 (2004) 71--94.
\bibitem[JS]{JS} D.~Joyce and Y.~Song, \textit{A theory of generalized {D}onaldson-{T}homas invariants}, Memoirs of the AMS (2012). 
\bibitem[Kly]{Kly} A.~A.~Klyachko, \textit{Vector bundles and torsion free sheaves on the projective plane}, preprint Max Planck Institut f\"ur Mathematik (1991).
\bibitem[Koo]{Koo} M.~Kool, \textit{Euler characteristics of moduli spaces of torsion free sheaves on toric surfaces}, Geom.~Ded.~176 (2015) 241--269.
\bibitem[Laa1]{Laa1} T.~Laarakker, \textit{Monopole contributions to refined Vafa-Witten invariants}, Geom.~Topol.~24 (2020) 2781--2828.
\bibitem[Laa2]{Laa2} T.~Laarakker, \textit{Vertical Vafa-Witten invariants}, arXiv:1906.01264.
\bibitem[LL]{LL} J.~M.~F.~Labastida and C.~Lozano, \textit{The Vafa-Witten theory for gauge group $\SU(N)$}, Adv.~Theor.~Math.~Phys.~5 (1999) 1201--1225.
\bibitem[LQ1]{LQ1} W.-P.~Li and Z.~Qin, \textit{On blowup formulae for the $S$-duality conjecture of Vafa and Witten}, Invent.~Math.~136 (1999) 451--482.
\bibitem[LQ2]{LQ2}  W.-P.~Li and Z.~Qin, \textit{On blowup formulae for the $S$-duality conjecture of Vafa and Witten II: the universal functions}, Math.~Res.~Lett.~5 (1998) 439--453.
\bibitem[Lie1]{Lie1} M.~Lieblich, \textit{Moduli of twisted sheaves}, Duke Math.~Jour.~138 (2007) 23--118. 
\bibitem[Lie2]{Lie2} M.~ Lieblich, \textit{Twisted sheaves and the period-index problem}, Compos.~Math.~144 (2008) 1--31.
\bibitem[Man]{Man} J.~Manschot, \textit{The Betti numbers of the moduli space of stable sheaves of rank 3 on $\mathbb{P}^2$}, Letters in Math.~Phys.~98 (2011) 65--78.  
\bibitem[MT]{MT} D.~Maulik and R.~P.~Thomas, \textit{Sheaf counting on local K3 surfaces}, Pure Appl.~Math.~Quart.~14 (2018) 419--441.
\bibitem[Moc]{Moc} T.~Mochizuki, \textit{Donaldson type invariants for algebraic surfaces}, Lecture Notes in Math.~1972, Springer-Verlag, Berlin (2009). 
\bibitem[Moz]{Moz} S.~Mozgovoy, \textit{Invariants of moduli spaces of stable sheaves on ruled surfaces}, arXiv:1302.4134.
\bibitem[OG]{OG} K.~O' Grady, \textit{The weight-two Hodge structure of moduli space of sheaves on a K3 surface}, J.~Alg.~Geom.~6 (1999) 599--644.
\bibitem[PTVV]{PTVV} T.~Pantev, B.~T\"{o}en, M.~Vaqui\'{e}, and G. Vezzosi, \textit{Shifted symplectic structures}, Publ.~Math.~I.H.E.S.~117 (2013) 271--328.
\bibitem[TT1]{TT1} Y.~Tanaka and R.~P.~Thomas, \textit{Vafa-Witten invariants for projective surfaces I: stable case}, Jour.~Alg.~Geom.~29 (2020) 603--668. 
\bibitem[TT2]{TT2} Y.~Tanaka and R.~P.~Thomas, \textit{Vafa-Witten invariants for projective surfaces II: semistable case}, Pure Appl.~Math.~Quart.~13 (2017) 517--562.
\bibitem[Tho]{Tho} R.~P.~Thomas, \emph{Equivariant K-theory and refined Vafa-Witten invariants}, Comm.~Math.~Phys.~378 (2020) 1451--1500.
\bibitem[Tod1]{Tod1} Y.~Toda, \textit{Stable pairs on local K3 surfaces}, J.~Diff.~Geom.~92 (2012) 285--370. 
\bibitem[Tod2]{Tod2} Y.~Toda, \textit{Stability conditions and curve counting invariants on Calabi-Yau 3-folds}, Kyoto Jour.~Math.~52 (2012) 1--50.
\bibitem[VW]{VW} C.~Vafa and E.~Witten, \textit{A strong coupling test of $S$-duality}, Nucl.~Phys.~B 431 (1994) 3--77.
\bibitem[Wei]{Wei} T.~Weist, \textit{Torus fixed points of moduli spaces of stable bundles of rank three}, J.~Pure Appl.~Algebra 215 (2011) 2406--2422.
\bibitem[Wit]{Wit} E.~Witten, \textit{AdS/CFT correspondence and topological field theory}, JHEP 9812 (1998) 012.
\bibitem[Yos1]{Yos1} K.~Yoshioka, \textit{The Betti numbers of the moduli space of stable sheaves of rank 2 on $\PP^2$}, J.~Reine Angew.~Math.~453 (1994) 193--220.
\bibitem[Yos2]{Yos2} K.~Yoshioka, \textit{The Betti numbers of the moduli space of stable sheaves of rank 2 on a ruled surface}, Math.~Ann.~302 (1995) 519--540.
\bibitem[Yos3]{Yos3} K.~Yoshioka, \textit{Some examples of Mukai's reflections on K3 surfaces},  J.~Reine~Angew.~Math.~515 (1999) 97--123. 
\bibitem[Yos4]{Yos4} K.~Yoshioka, \textit{Moduli spaces of twisted sheaves on a projective variety}, in: Moduli spaces and arithmetic geometry (Kyoto, 2004), Adv.~Stud.~Pure Math.~45 (2006) 1--42.
\end{thebibliography}
\end{document}